\documentclass[reqno, 12pt]{amsart}
\usepackage[letterpaper, margin=1in]{geometry}

\usepackage{amsmath, amssymb, amsfonts, amsthm, mathtools} 
\usepackage{enumitem} 
\setenumerate{label = \textup{(\textit{\roman*})}}
\usepackage[dvipsnames]{xcolor}
\usepackage{tikz}
\usetikzlibrary{calc,decorations.markings,arrows.meta}
\usepackage{tkz-graph}
\usepackage{comment}
\usepackage{bm}
\usepackage{upgreek}
\usepackage[english]{babel}
\usepackage{csquotes}
\usepackage{tikz-cd}
\usepackage{subcaption}
\usepackage{float}
\usepackage{hyperref}
\usepackage{microtype}

\allowdisplaybreaks[3]

\hypersetup{
    linkcolor = blue!80!gray,
    citecolor = blue!80!gray,
    urlcolor = blue!80!gray,
    colorlinks = true,
}
\usepackage[noabbrev, capitalize, nameinlink]{cleveref}

\newcommand{\colorA}{Orchid}
\newcommand{\colorB}{Tan}

\newcommand{\Pic}{\operatorname{Pic}}
\newcommand{\vol}{\operatorname{vol}}

\theoremstyle{definition}
\newtheorem{theorem}{Theorem}[section]
\newtheorem{definition}[theorem]{Definition}
\newtheorem{lemma}[theorem]{Lemma}
\newtheorem{proposition}[theorem]{Proposition}
\newtheorem{corollary}[theorem]{Corollary}

\newtheorem{example}[theorem]{Example}

\newtheorem{remark}[theorem]{Remark}

\tikzset{
  mynode/.style={fill,circle,draw,inner sep=2pt,outer sep=0pt}
}

\newcommand{\R}{\mathbb{R}}

\newcommand{\ZZ}{\mathbb{Z}}
\newcommand{\ZZZ}{\mathbb{Z}_{\ge 0}}

\newcommand{\C}{\mathbb{C}}

\newcommand{\id}{\mathrm{id}}

\newcommand{\supp}{\mathrm{supp}}

\usepackage{dsfont}

\title{Postnikov--Stanley Polynomials are Lorentzian}
\date{\today}

\author{Serena An}
\address{Massachusetts Institute of Technology}
\email{anser@mit.edu}

\author{Katherine Tung}
\address{Harvard University}
\email{katherinetung@college.harvard.edu}

\author{Yuchong Zhang}
\address{University of Michigan}
\email{zongxun@umich.edu}

\begin{document}

\begin{abstract}
Postnikov--Stanley polynomials $D_u^w$ are a generalization of skew dual Schubert polynomials to the setting of arbitrary Weyl groups. We prove that Postnikov--Stanley polynomials are Lorentzian by showing that they are degree polynomials of Richardson varieties. Our result yields an interesting class of Lorentzian polynomials related to the geometry of Richardson varieties, generalizes the result that dual Schubert polynomials are Lorentzian (Huh--Matherne--M\'esz\'aros--St.\ Dizier 2022), and resolves the conjecture that Postnikov--Stanley polynomials have M-convex support (An--Tung--Zhang 2024). 
\end{abstract}

\maketitle

\section{Introduction}
Consider a Weyl group $W$ generated by the simple reflections $s_{\alpha_1},s_{\alpha_2},\dots,s_{\alpha_r}$ corresponding to the simple roots $\alpha_1,\dots,\alpha_r$. The Postnikov--Stanley polynomials $\{D_u^w\}_{u\le w\in W}$ \cite{PostnikovStanley} are defined as follows. For any $\lambda$ in the weight lattice $\Lambda$, let $x_i$ denote the inner product $(\lambda,\alpha_i^{\vee})$ induced by the Killing form, where $\alpha^{\vee}$ is the coroot $\frac{2\alpha}{(\alpha,\alpha)}$ corresponding to the root $\alpha\in \Phi$. Equivalently, we have $\lambda=x_1 \omega_1+\dots+x_r \omega_r$, where $\omega_1, \dots, \omega_r$ are the fundamental weights.
For a reflection $s_{\alpha}$ and a covering relation $u \lessdot us_{\alpha}$ in the strong Bruhat order of $W$, the \emph{Chevalley multiplicity} is defined by
\[m(u \lessdot us_{\alpha}) \coloneq (\lambda, \alpha^{\vee})=\sum_{i=1}^r c_i (\lambda, \alpha_i^{\vee})=\sum_{i=1}^r c_i x_i,\] 
where $c_i$ are nonnegative integers determined by $\alpha$ and the Lie type of $W$. Let the saturated chain $C = (u_0 \lessdot u_1\lessdot\cdots \lessdot u_{\ell})$ have weight 
\[m_C\coloneq m(u_0 \lessdot u_1)m(u_1 \lessdot u_2)\cdots m(u_{\ell - 1} \lessdot u_{\ell}).\]

\begin{definition}(\cite{PostnikovStanley})\label{def:PS}
    Given a Weyl group $W$ of a rank $r$ root system and elements $u,w \in W$, the \emph{Postnikov--Stanley polynomial} $D_u^w\in \R[x_1,\dots,x_r]$ is defined by \[D_u^w(x_1, \dots,x_{r}) \coloneqq \frac{1}{(\ell(w)-\ell(u))!} \sum_C m_C(x_1, \dots,x_{r}),\] where $\ell(w)$ and $\ell(u)$ denote the Coxeter length of $w$ and $u$ respectively, and the sum is over all saturated chains $C$ from $u$ to $w$ in the strong Bruhat order of $W$. As an abuse of notation, we also write $D_u^w(\lambda)\coloneqq D_u^w(x_1, \dots,x_{r})$, where $\lambda=x_1 \omega_1+\dots+x_r\omega_r$ for fundamental weights $\omega_1,\dots,\omega_r$.
\end{definition}

Richardson varieties were introduced by Kazhdan and Lusztig \cite{KLPonincare} in 1980 and named after Richardson in honor of his work \cite{Richardson} in 1992. Historically, Richardson varieties in Grassmannians can be traced back to Hodge \cite{Hodge} in 1942. Richardson varieties play an essential role in understanding intersection theory of flag varieties and have been linked with algebraic geometry, representation theory, and combinatorics. For instance, Deodhar \cite{Deodhar} explored their relationship with Kazhdan--Lusztig theory, Brion and Lakshmibai \cite{standardmonomial} advanced standard monomial theory, and Marsh and Rietsch \cite{nonnegativeflag} studied totally nonnegative parts of flag varieties. Further applications of Richardson varieties span several domains, including equivariant {$K$}-theory \cite{RichardsonK}, order polytopes \cite{bruhatintervalpolytopes,toricbruhat} and Catalan combinatorics \cite{richardsoncatalan}. In recent years, Richardson varieties have attracted increasing interest \cite{singgenRich,singRich,KawamataRich,torusquotient,toricrich,projRich,quantumRich,relRich}. Here, we work with \emph{closed} Richardson varieties $R_u^w$, which are the scheme-theoretical intersection $X^w\cap X_u$ of a Schubert variety $X^w$ and an opposite Schubert variety $X_u$. 

In this paper, we establish a connection between Postnikov--Stanley polynomials and the geometry of Richardson varieties. We show that the $\lambda$-degree of a Richardson variety $R_u^w$ can be expressed as $(\ell(w) - \ell(u))! \cdot D_u^w(\lambda)$. The $\lambda$-degree counts the number of points in the intersection of $e(R_u^w)$ and a generic linear subspace of complex codimension $\ell(w) - \ell(u)$ in $\mathbb{P}(V_{\lambda})$, where $V_{\lambda}$ is the irreducible representation of the Lie group $G$ corresponding to the dominant weight $\lambda\in \Lambda^+$, and $e: G/B \to \mathbb{P}(V_{\lambda})$ is the Borel--Weil mapping of the flag variety $G/B$. Furthermore, we show that the volume of the Newton--Okounkov body $\Delta_{w_0}(R_u^w, \lambda)$ is also given by the Postnikov--Stanley polynomial $D_u^w(\lambda)$.

This geometric perspective helps us contribute to the theory of Lorentzian polynomials \cite{Lorentzian}. The theory of {completely log-concave polynomials}, which when homogeneous are exactly the Lorentzian polynomials, was simultaneously and independently developed in a series of papers beginning with \cite{logconcavepoly}, but we use the term ``Lorentzian" in this paper for concision. Inspired by Hodge's index theorem for projective varieties, in 2019, Br\"and\'en and Huh defined Lorentzian polynomials, which are a class of polynomials with deep applications in combinatorics \cite{Lorentzian}. They developed a theory around these polynomials that enabled them and many others to prove a series of long-standing conjectures about the log-concavity and unimodality of coefficients of specializations of the Tutte polynomial \cite{Mas72, Her72, Rot71, Welsh, Bry82}. Typical examples of Lorentzian polynomials include homogeneous stable polynomials and volume polynomials of convex bodies and projective varieties \cite{Lorentzian}. Within algebraic combinatorics, notable examples are dual Schubert and normalized Schur polynomials \cite{HuhMeszaroLorSchur}, multivariate Tutte polynomials of morphisms of matroids for sufficiently small parameters \cite{EurHuh}, and normalizations of certain projections of integer point transforms of flow polytopes \cite{MesSet}. In 2024, Nadeau, Spink, and Tewari showed that toric Richardson varieties have volume polynomials equal to a family of degree polynomials associated to skew Schur polynomials indexed by ribbon shapes \cite{tewari}. These volume polynomials are Lorentzian, giving another interesting family of Lorentzian polynomials. Lorentzian polynomials also have M-convex support, which is equivalent to having a generalized permutahedron as its Newton polytope and having a saturated Newton polytope. This latter property is interesting in its own right, sparking much recent interest \cite{NewtonPolyinAC, FMD, CCMM, whole}. 

Building on work of Br\"and\'en and Huh \cite{Lorentzian} regarding volume polynomials, we prove the following result.

\begin{theorem}\label{theorem:PS Lorentzian}
    For any Weyl group $W$ and elements $u\le w\in W$, the Postnikov--Stanley polynomial $D_u^w$ is Lorentzian.
\end{theorem}

This theorem generalizes \cite[Proposition~18]{HuhMeszaroLorSchur}: by letting $W$ be of type A and $u = \id$, we obtain that dual Schubert polynomials are Lorentzian. The following is a straightforward corollary of \Cref{theorem:PS Lorentzian}. 

\begin{corollary}(cf.\ \cite[Conjecture~5.1]{whole})\label{cor:PS M-conv}
Postnikov--Stanley polynomials have \emph{M-convex} support; equivalently, their Newton polytopes are saturated and are generalized permutahedra.
\end{corollary}

This paper is organized as follows. In \cref{section:preliminaries}, we give background and definitions. In \cref{section:degree}, we prove that Postnikov--Stanley polynomials are degree polynomials of Richardson varieties. In \cref{section:Lorentzian}, we prove that Postnikov--Stanley polynomials are Lorentzian.

\section*{Acknowledgments}
This work is a continuation of research conducted under the mentorship of Shiyun Wang and Meagan Kenney at the University of Minnesota Combinatorics and Algebra REU. We thank Shiliang Gao for his guidance. We also thank Grant Barkley, Christian Gaetz, Yibo Gao, Joe Harris, June Huh, Pavlo Pylyavskyy, Vasu Tewari, and Lauren Williams for helpful discussions.

\section{Preliminaries}\label{section:preliminaries}
\subsection{Flag varieties}
In this paper, we use standard terminology in Weyl groups and Schubert calculus. The readers are referred to \cite[Chapter~4]{linearalggrp} and \cite[Chapter~2]{Schubdegloci} for a more detailed exposition.

Let $G/B$ be a \emph{flag variety}, where $G$ is a complex semisimple simply-connected Lie group, and $B$ is a Borel subgroup containing a maximal torus $T$. Let $B_-$ be the opposite Borel subgroup of $G$ determined by $B \cap B_- = T$. The Weyl group $W$ is defined by $W\cong N(T)/T$, where $N(T)$ is the normalizer of $T$ in $G$. 

Let $\Phi$ be the root system associated with $W$, and let $\alpha_1,\dots,\alpha_r$ be the simple roots in $\Phi$. Define $V$ as the $\R$-linear space spanned by $\alpha_1,\dots,\alpha_r$, equipped with an inner product $(\cdot,\cdot)$ induced by the Killing form. For each root $\alpha\in \Phi$, the corresponding \emph{coroot} is defined by $\alpha^{\vee}\coloneqq \frac{2\alpha}{(\alpha,\alpha)}$.

The \emph{fundamental weights} $\omega_1,\dots,\omega_r\in V$ are defined by the condition $(\omega_i,\alpha_j^{\vee})=\delta_{i,j}$, where $\delta_{i,j}$ is the Kronecker delta. The \emph{weight lattice} $\Lambda$ is given by \[\Lambda\coloneqq\{x_1\omega_1+\dots+x_r\omega_r\mid x_1,\dots,x_r\in \mathbb{Z}\},\] and the subset of \emph{dominant weights} $\Lambda^+$ is \[\Lambda^+\coloneqq\{x_1\omega_1+\dots+x_r\omega_r\mid x_1,\dots,x_r\in \mathbb{Z}_{\ge0}\}.\]
For the weight lattice $\Lambda$, we have the following theorem.

\begin{theorem}(\cite[Proposition~2.3]{Tignol1995})\label{thm:picweight}
    There is an isomorphism of abelian groups between the Picard group $\Pic(G/B)$ and the weight lattice $\Lambda$ under vector addition.
\end{theorem}
We use $\mathcal{L}_{\lambda}$ to denote the line bundle in $\Pic(G/B)$ corresponding to the weight $\lambda\in \Lambda$.

\subsection{Bruhat order}
In this paper, we use standard terminology for the (strong) Bruhat order. Readers are referred to \cite[Chapter~2]{BjornerBrenti} for a more detailed exposition.

Let $W$ be a Weyl group. Each element $w \in W$ can be expressed as a product of \emph{simple reflections} $s_{\alpha_i}$. 
If $w = s_{i_1}\cdots s_{i_{\ell}}$ is expressed as a product of simple reflections with $\ell$ minimal among all such expressions, then the string $s_{i_1}\cdots s_{i_{\ell}}$ is called a \emph{reduced decomposition} of $w$. We call $\ell \coloneqq \ell(w)$ the \emph{(Coxeter) length} of $w$.

\begin{definition}(\cite{BjornerBrenti})\label{def:bruhatorder}
    The \emph{(strong) Bruhat order} $\le$ on a Weyl group $W$ is defined as follows. Let $u, v\in W$ and $\ell = \ell(v)$. 
    We have $u\le v$ if and only if for any reduced decomposition $s_{i_1}\cdots s_{i_{\ell}}$ of $v$, there exists a reduced decomposition $s_{j_1}\cdots s_{j_k}$ of $u$ such that $s_{j_1}\cdots s_{j_k}$ is a (not necessarily consecutive) substring of $s_{i_1}\cdots s_{i_{\ell}}$. If additionally $\ell(v) = \ell(u) + 1$, we write $u\lessdot v$. 
\end{definition}

For $u\le w$ in the Bruhat order of $W$, the \emph{(Bruhat) interval} $[u, w]$ is the subposet containing all $v\in W$ such that $u\le v\le w$. 

\subsection{Postnikov--Stanley polynomials}
We defined Postnikov--Stanley polynomials $D_u^w$ for $u\le w \in W$ in \cref{def:PS}. The \emph{skew dual Schubert polynomials} are Postnikov--Stanley polynomials for which $W$ is of type A, and the \emph{dual Schubert polynomials} are Postnikov--Stanley polynomials for which $W$ is of type A and $u = \id$.

\begin{example}\label{ex:PS}
    Let $W = A_2$, and let $u = 213$ and $w = 321$. In the Bruhat order of $W$, there are two saturated chains from $u$ to $w$, namely $213\lessdot 231\lessdot 321$ and $213\lessdot 312\lessdot 321$, as shown in \cref{fig:123to321}. The first chain has weight $x_1x_2$ and the second chain has weight $(x_1 + x_2) \cdot x_2$. Thus, $D_{213}^{321} = \frac{1}{2!}(x_1 x_2 + (x_1 + x_2)\cdot x_2)$. Note that $D_{213}^{321}$ is a Postnikov--Stanley polynomial and a skew dual Schubert polynomial, but not a dual Schubert polynomial.
\end{example}

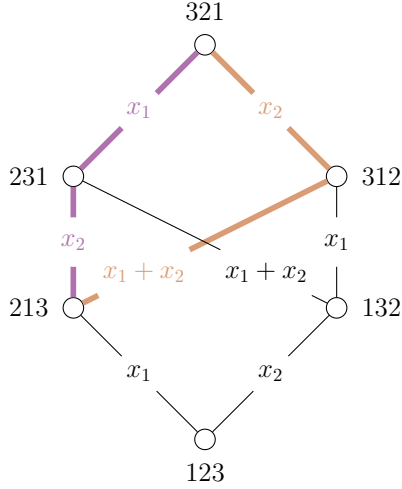
\begin{figure}
    \centering
    \begin{tikzpicture}[auto=center,every node/.style={circle, fill=white,scale=0.7,draw=black, solid}, scale=1.75, label distance=0.5mm]
        \node (a) at (0, 0) {};
        \node[label=left:{\large $213$}] (b) at (-1, 1) {};
        \node[label=right:{\large $132$}] (d) at (1, 1) {};
        \node[label=left:{\large $231$}] (e) at (-1, 2) {};
        \node[label=right:{\large $312$}] (g) at (1, 2) {};
        \node (h) at (0, 3) {};

        \node[draw=none, fill=none] at ($(a)-(0, 0.25)$) {\large $123$};
        \node[draw=none, fill=none] at ($(h)+(0, 0.25)$) {\large $321$};
        
        \draw (a)--node [draw=none] {\large $x_1$}(b);
        \draw (a)--node [draw=none] {\large $x_2$}(d);
        \draw[line width=0.75mm, \colorA] (b)--node [draw=none] {\large $x_2$}(e);
        \draw[line width=0.75mm, \colorA] (e)--node [draw=none] {\large $x_1$}(h);
        \draw[line width=0.75mm, \colorB] (g)--node [draw=none] {\large $x_2$}(h);
        \draw (d)--node [draw=none] {\large $x_1$}(g);
        \draw[line width=0.75mm, \colorB] (b)--node [near start, draw=none] {\large $x_1 + x_2$}(g);
        \draw (d)--node [near start, draw=none] {\large $x_1 + x_2$}(e);
    \end{tikzpicture}
    \caption{Calculating the Postnikov--Stanley polynomial $D_{213}^{321}$ in $A_2$.}
    \label{fig:123to321}
\end{figure}
\subsection{Schubert varieties and Richardson varieties}The \emph{Schubert variety} $X^w$ in $G/B$ is defined as the $B$-orbit closure $\overline{BwB/B}\subseteq G/B$, and the \emph{opposite Schubert variety} $X_w$ in $G/B$ is defined as the $B_-$-orbit closure $\overline{B_-wB/B}\subseteq G/B$. The \emph{Schubert class} $\sigma_w$ is defined in the cohomology ring $ H^*(G/B)$ as $\sigma_w\coloneqq[X^{w_0w}]$, where $w_0$ is the element of maximal Coxeter length in the Weyl group $W$.

A \emph{Richardson variety}, defined as the intersection of a Schubert variety and an opposite Schubert variety, is given as follows.

\begin{definition}(\cite{KLPonincare})
    For $u \le w$ in the Bruhat order of a Weyl group $W$, the \emph{(closed) Richardson variety} is the intersection $X^w \cap X_u$ within $G/B$.
\end{definition}

Since the intersection $X^w \cap X_u$ is transverse, we have the following proposition.

\begin{proposition}(\cite{SpeyerRich})\label{prop:rich class}
    Given a Richardson variety $R_u^w$, the Richardson class $[R_u^w]\in H^*(G/B)$ is given by the formula $[R_u^w]=[X^w]\cdot[X_u]=\sigma_{w_0w}\cdot\sigma_{u}$.
\end{proposition}

\subsection{Lorentzian polynomials} In this section, we define Lorentzian polynomials and state a crucial theorem for proving \cref{theorem:PS Lorentzian}.

\begin{definition}(\cite{DicreteConvexAnaly})
A subset $\mathcal{J} \subset \ZZ^n$ is \emph{M-convex} if for any index $i \in [n]$ and any $\alpha, \beta \in \mathcal{J}$ whose $i$th coordinates satisfy $\alpha_i > \beta_i$, there is an index $j \in [n]$ satisfying
\[
\alpha_j < \beta_j \quad \text{and} \quad \alpha - e_i + e_j \in \mathcal{J} \quad \text{and} \quad \beta - e_j + e_i \in \mathcal{J}.
\]
    A polynomial $f = \sum_{\alpha \in \mathbb{Z}_{\ge 0}^{n}} c_{\alpha} x^{\alpha} \in \R[x_{1}, \ldots, x_{n}]$ is \emph{M-convex} if its \emph{support} $\supp(f)\coloneqq \left\{\alpha\in \ZZZ^{n}\mid c_{\alpha} \neq 0\right\}$ is an M-convex set.
\end{definition}
\begin{definition}(\cite[Definition~5]{HuhMeszaroLorSchur}; see also \cite{Lorentzian})\label{def:Lorentzian}
    Let $h(x_1, \ldots, x_n)$ be a degree $d$ homogeneous polynomial, and let $e=d-2$. We say that $h$ is \emph{strictly Lorentzian} if all the coefficients of $h$ are positive and the quadratic form $\frac{\partial}{\partial x_{i_1}} \cdots \frac{\partial}{\partial x_{i_e}}h$  has the signature $(+, -, \ldots, -)$ for any $i_1, \ldots i_e \in [n]$. 
    
    We say that $h$ is \emph{Lorentzian} if it satisfies any one of the following equivalent conditions.
    \begin{enumerate}
        \item[(1)] All the coefficients of $h$ are nonnegative, the support of $h$ is M-convex, and the quadratic form $\frac{\partial}{\partial x_{i_1}} \cdots \frac{\partial}{\partial x_{i_e}}h$ has at most one positive eigenvalue for any $i_1, \ldots i_e \in [n]$.
        \item[(2)] All the coefficients of $h$ are nonnegative and, for any $i_1, i_2, \ldots \in [n]$ and any positive $k$, the functions $h$ and $\frac{\partial}{\partial x_{i_1}} \cdots \frac{\partial}{\partial x_{i_k}}h $ are either identically zero or log-concave on $\R^n_{>0}$.
        \item[(3)] The polynomial $h$ is a limit of strictly Lorentzian polynomials.
    \end{enumerate}
\end{definition}

Recall that given a projective variety $X$, a line bundle $\mathcal{L}$ on $X$ is \emph{numerically effective (nef)} if $\int_C c_1(\mathcal{L})\ge 0$ for all irreducible curves $C\subseteq X$, where $c_1(\mathcal{L})$ is the first Chern class of $\mathcal{L}$, and $\int_C: H^2(X,\mathbb{Z})\rightarrow \mathbb{Z}$ is the map that evaluates a cohomology class on the fundamental class of $C$ in $H_2(X,\mathbb{Z})$. A Cartier divisor $D$ on $X$ is \emph{nef} if the associated line bundle $\mathcal{O(D)}$ is nef. For further details, see \cite[Chapter~1]{PositivityAG1}.

\begin{theorem}(\cite[Theorem~4.6]{Lorentzian})\label{thm:volume Lorentzian}
    Let $\mathbf{H}=(H_1,\dots,H_r)$ be a collection of nef Cartier divisors on an $\ell$-dimensional irreducible projective variety $X$ over an algebraically closed field. Then the volume polynomial 
    \[\vol_{X,\mathbf{H}}(x_1,\dots,x_r)\coloneqq (x_1H_1+x_2H_2+\dots+x_rH_r)^{\ell}\] is Lorentzian.
\end{theorem}

\section{Degree polynomials of Richardson varieties}\label{section:degree}
In this section, we prove that Postnikov--Stanley polynomials are degree polynomials of Richardson varieties. Before defining degree polynomials, we recall the definition of a \emph{Poincaré pairing}.

\begin{definition}
Let $X$ be a smooth projective variety over $\C$ of dimension $n$, and let $H^i(X, \mathbb{Z})$ denote the $i$th cohomology group with integer coefficients. The \emph{Poincaré pairing} is a non-degenerate bilinear form
\[
\langle \cdot, \cdot \rangle : H^i(X, \mathbb{Z}) \times H^{2n-i}(X, \mathbb{Z}) \to \mathbb{Z}
\]
defined by the composition of the cup product $\cdot: H^i(X, \mathbb{Z}) \times H^{2n-i}(X, \mathbb{Z})\to H^{2n}(X, \ZZ)$ and evaluation $\int_X : H^{2n}(X, \mathbb{Z}) \to \mathbb{Z}$ of a cohomology class on the fundamental class of $X$ in $H_{2n}(X, \ZZ)$; that is,
\[
\langle \theta_1,\theta_2 \rangle = \int_X \theta_1 \cdot \theta_2.
\]
\end{definition}

We restrict our attention to the case of $X=G/B$ a flag variety and work with the cohomology ring $H^*(G/B)$ expressed in the basis of Schubert classes $\{\sigma_w:w\in W\}$. In this context, the Poincaré pairing $\langle \theta_1,\theta_2\rangle$ is the coefficient of the top-degree Schubert class $\sigma_{w_0}$ in the expansion of $\theta_1\cdot \theta_2$ in terms of Schubert classes, where $\sigma_{w_0}$ corresponds to the longest element $w_0$ in the Weyl group $W$.

Let $\lambda$ denote the weight $x_1\omega_1+\dots+x_r\omega_r\in \Lambda$, and let $\mathcal{L}_{\lambda}$ be the line bundle associated with $\lambda$ from \Cref{thm:picweight}. Let $\overline{\lambda}$ denote the first Chern class $c_1(\mathcal{L}_\lambda)$ of $\mathcal{L}_{\lambda}$.

\begin{definition}\label{def:degreepoly}
    The \emph{$\lambda$-degree} $\deg_{\lambda}(X)$ of an $\ell$-dimensional irreducible subvariety $X$ of $G/B$ is defined as the Poincar\'e pairing $\langle \overline{\lambda}^{\ell}, [X]\rangle$.
\end{definition}

Let $D$ denote the Cartier divisor associated with the line bundle $\mathcal{L}_\lambda$. The $\lambda$-degree $\langle \overline{\lambda}^{\ell}, [X]\rangle$ represents the intersection product $(D^{\ell}\cdot X)$, which corresponds to the self-intersection number $\int_{X} (D)^{\ell}\in \mathbb{Z}$. For further details, see \cite[Chapter~1]{PositivityAG1}.

The main result of this section is that a Postnikov-Stanley polynomial evaluated at $\lambda$ gives the $\lambda$-degree of the corresponding Richardson variety up to multiplication by a scalar.
\begin{proposition}\label{prop:degree}
    Given $u,w\in W$ and $\lambda\in \Lambda^+$, the $\lambda$-degree $\deg_\lambda(R_u^w)$ of the Richardson variety $R_u^w$ is given by 
    \[\deg_\lambda(R_u^w)=(\ell(w)-\ell(u))!\cdot D_u^w(\lambda).\]
\end{proposition}

Our \cref{prop:degree} is a direct generalization of \cite[Proposition 4.2]{PostnikovStanley} from the setting of Schubert varieties to that of Richardson varieties, and is proved along similar lines. Before presenting the proof, we give an example, the {Chevalley formula}, and a few technical lemmas.

\begin{example}\label{ex:deg}
    Continuing \cref{ex:PS}, let $W = A_2$, $u = 213$, and $ w = 321$. Consider the simple roots $\alpha_1 = (1,-1,0)$ and $\alpha_2 = (0,1,-1)$, and let $V = \mathrm{span}_{\R}(\alpha_1,\alpha_2)$. We obtain fundamental weights $\omega_1 = (\frac{2}{3}, -\frac{1}{3}, -\frac{1}{3})$ and $\omega_2 = (\frac{1}{3}, \frac{1}{3}, -\frac{2}{3})$ in $V$. Let $\lambda$ be the dominant weight $\omega_1 + \omega_2 = (1,0,-1)$, so $x_1 = x_2 = 1$. Then \[\deg_{\lambda} R_u^w = (\ell(w) - \ell(u))! \cdot D_u^w(1,1) = 2! \cdot \frac{3}{2} = 3.\]
\end{example}

\begin{lemma}(\cite{Chevalley})\label{lemma:chevalley}
Given a Weyl group $W$ and a Schubert class $\sigma_w$, we have
    \[\overline{\lambda}\cdot \sigma_w=\sum_{\alpha} (\lambda,\alpha^{\vee})\sigma_{ws_{\alpha}}=\sum_{\alpha} m(w\lessdot ws_{\alpha})\sigma_{ws_{\alpha}},\]
    where the sum is over all roots $\alpha\in\Phi$ such that $\ell(ws_{\alpha})=\ell(w)+1$.
\end{lemma}

As a corollary, we have the following result.

\begin{lemma}\label{lemma:chainsum}
    Given a Weyl group $W$ and a Schubert class $\sigma_u$, we have  for any positive integer $\ell$ that
    \[\overline{\lambda}^\ell \cdot \sigma_u=\ell!\cdot\sum_{w' \in W: \ell(w')=\ell(u)+\ell}D_u^{w'}(\lambda) \sigma_{w'}. \]
\end{lemma}
\begin{proof}
    By \Cref{def:PS}, we wish to show \[\overline{\lambda}^\ell \cdot \sigma_u=\sum_{w'}\sum_{u=u_0\lessdot\dots\lessdot u_\ell =w'} \prod_{i=1}^{\ell} m(u_{i-1}\lessdot u_i) \sigma_{w'}.\]
    We use induction on $\ell$, with the base case $\ell=1$ following from \Cref{lemma:chevalley}. For the inductive step, suppose the result holds for $\ell-1$. Then,
    \begin{align*}
        \overline{\lambda}^\ell \cdot \sigma_u&=\sum_{v}\sum_{u=u_0\lessdot\dots\lessdot u_{\ell-1} =v} \prod_{i=1}^{\ell-1} m(u_{i-1}\lessdot u_i) (\overline{\lambda}\cdot\sigma_{v})\\
        &=\sum_{v}\sum_{u=u_0\lessdot\dots\lessdot u_{\ell-1} =v} \prod_{i=1}^{\ell-1} m(u_{i-1}\lessdot u_i) \Big(\sum_{v\lessdot w'}m(v\lessdot w')\sigma_{w'}\Big) & \text{ \Cref{lemma:chevalley}}\\
        &=\sum_{w'}\sum_v\sum_{u=u_0\lessdot\dots\lessdot u_{\ell-1}=v\lessdot w'} \left(\prod_{i=1}^{\ell-1} m(u_{i-1}\lessdot u_i)\right)\cdot m(v\lessdot w')\sigma_{w'}\\
        &=\sum_{w'}\sum_{u=u_0\lessdot\dots\lessdot u_\ell =w'} \prod_{i=1}^{\ell} m(u_{i-1}\lessdot u_i) \sigma_{w'},
    \end{align*}
   as desired.
\end{proof}

The next lemma will help us compute the Poincar\'e pairing.

\begin{lemma}(\cite{PostnikovStanley})\label{lemma:poincaredual}
    Given a Weyl group $W$ and Schubert classes $\sigma_u$ and $\sigma_{w_0w}$, we have $\langle\sigma_u, \sigma_{w_0 w}\rangle=\delta_{u,w}$, where $\delta_{u,w}$ is the Kronecker delta.
\end{lemma}

\begin{proof}[Proof of \Cref{prop:degree}]
    Let $\ell\coloneq \ell(w)-\ell(u)$. By \Cref{def:degreepoly} and \Cref{prop:rich class}, we have 
    \[\deg_{\lambda}(R_u^w)=\langle [R_u^w], \overline{\lambda}^{\ell}\rangle = \langle [X^w]\cdot[X_u],\overline{\lambda}^{\ell}\rangle=\langle \sigma_{w_0 w}\cdot \sigma_u, \overline{\lambda}^{\ell}\rangle,\]
    which corresponds to the coefficient of $\sigma_{w_0}$ in $\sigma_{w_0 w}\cdot \sigma_u \cdot \overline{\lambda}^{\ell}$.
    By \Cref{lemma:chainsum}, we have $\sigma_u \cdot \overline{\lambda}^{\ell}=\ell!\cdot\sum_{w'}D_u^{w'}(\lambda) \sigma_{w'}$, where the sum is over all $w'\in W$ such that $\ell(w')=\ell(u)+\ell$. Therefore, \[\deg_{\lambda}(R_u^w)=\ell!\cdot\sum_{w'} D_u^{w'}(\lambda) \langle\sigma_{w_0w}, \sigma_{w'}\rangle=\ell!\cdot D_u^w(\lambda),\]
    where the last equality follows from \Cref{lemma:poincaredual}.
\end{proof}

\begin{remark}
    As a corollary of \Cref{prop:degree} and  \cite[Theorem~2.5]{polytopeXL} (see also \cite[Remark~3.9(iii)]{Toricsperical}), Postnikov--Stanley polynomials $D_u^w(\lambda)$ represent the volume of the \emph{Newton--Okounkov body} $\Delta_{w_0}(R_u^w, \lambda)$ with respect to the Lebesgue measure of dimension $\ell(w)-\ell(u)$. For the definition of $\Delta_{w_0}(R_u^w, \lambda)$, see, for example, \cite[Section~2.5]{polytopeXL}.
\end{remark}

\section{Postnikov--Stanley polynomials are Lorentzian}\label{section:Lorentzian}
In this section, we prove \Cref{theorem:PS Lorentzian}. We require the following lemma.

\begin{lemma}(\cite[Proposition~2.2.10]{Hague2007CohomologyOF})\label{lemma:nef line bundle}
    Let $G/B$ be a flag variety, $\lambda\in \Lambda$ a weight, and $\mathcal{L}_\lambda$ the corresponding line bundle. Then $\mathcal{L}_\lambda$ is a nef line bundle if and only if $\lambda\in \Lambda^+$.
\end{lemma}

\begin{proof}[Proof of \Cref{theorem:PS Lorentzian}]
    By \Cref{lemma:nef line bundle}, the line bundle $\mathcal{L}_{\omega_i}$ corresponding to any fundamental weight $\omega_i\in \Lambda^+$ is a nef line bundle. Then the Cartier divisor $D_i$ corresponding to the first Chern class $\overline{\omega_i}\coloneqq c_1(\mathcal{L}_{\omega_i})$ is a nef Cartier divisor for each $i$. Let $D$ denote the Cartier divisor corresponding to $c_1(\mathcal{L}_{\lambda})$.
    Since $\lambda=\sum_{i=1}^r x_i \omega_i$, we have $\overline{\lambda}=\sum_{i=1}^r x_i \overline{\omega_i}$ and $D=\sum_{i=1}^r x_iD_i$. 

    By \Cref{prop:degree} and \Cref{def:degreepoly}, we have\[(\ell(w)-\ell(u))!\cdot D_u^w(\lambda)= \deg_{\lambda}(R_u^w)=\langle(c_1(\mathcal{L}_\lambda))^{\ell(w)-\ell(u)},[R_u^w]\rangle=\int_{G/B}(c_1(\mathcal{L}_\lambda))^{\ell(w)-\ell(u)}\cdot[R_u^w].\]
    Moreover, by \cite[Proposition~1.31]{3264}, we have
    \[\int_{G/B}(c_1(\mathcal{L}_\lambda))^{\ell(w)-\ell(u)}\cdot[R_u^w]=(D^{\ell(w)-\ell(u)}\cdot R_u^w) =\int_{G/B}(c_1(\mathcal{L}_\lambda|_{R_u^w}))^{\ell(w)-\ell(u)}.\]

    Let $D_i'\coloneqq D_i\cap R_u^w$ be the Cartier divisor corresponding to the line bundle $\mathcal{L}_{\omega_i}|_{R_u^w}$, and let $D'\coloneqq D\cap R_u^w$ be the Cartier divisor corresponding to the line bundle $\mathcal{L}_{\lambda}|_{R_u^w}$. By \cite[Example~1.4.4]{PositivityAG1}, since each $\mathcal{L}_{\omega_i}$ is a nef line bundle, each $\mathcal{L}_{\omega_i}|_{R_u^w}$ is a nef line bundle, and thus each $D_i'$ is a nef Cartier divisor. Moreover, we have
    \[\int_{G/B}(c_1(\mathcal{L}_\lambda|_{R_u^w}))^{\ell(w)-\ell(u)}=
 (D')^{\ell(w)-\ell(u)}= \left(\sum_{i=1}^r x_iD_i'\right)^{\ell(w)-\ell(u)}=
 \vol_{R_u^w,\mathbf{D'}}(x_1,\dots,x_r),\]
    where $\mathbf{D'}\coloneqq(D_1',\dots,D_r')$, and the last equality is by \Cref{thm:volume Lorentzian}. Since
    \[D_u^w(\lambda)=\frac{1}{(\ell(w)-\ell(u))!}\cdot \vol_{R_u^w,\mathbf{D'}}(x_1,\dots,x_r)\]
    and a nonzero multiple of a volume polynomial is Lorentzian, we obtain the desired result.
\end{proof}

\begin{proof}[Proof of \Cref{cor:PS M-conv}]
    This follows from \Cref{theorem:PS Lorentzian} and \Cref{def:Lorentzian} (1).
\end{proof}

\bibliographystyle{plain} 
\bibliography{references}

\end{document}